\documentclass[11pt]{article}
\usepackage{amsmath, amsfonts, amssymb, mathtools,amsthm,amsxtra}
\usepackage{graphicx, amsthm,color}
\usepackage{pgfplots, caption, subcaption, tikz, tikz-cd} 
\usepackage{float}
\usepackage{dsfont}
\usepackage{comment}

\newtheorem{theorem}{Theorem}
\newtheorem*{theorem*}{Theorem}
\newtheorem{corollary}[theorem]{Corollary}

\newtheorem{lemma}[theorem]{Lemma}
\newtheorem{proposition}[theorem]{Proposition}

\newtheorem{example}[theorem]{Example}
\newtheorem{definition}[theorem]{Definition}
\theoremstyle{definition}
\newtheorem{remark}[theorem]{Remark}

\begin{document}

\begin{center}
{\LARGE A functional characterization of isometries between non-reversible Finsler manifolds}

\bigskip

\textsc{Francisco Venegas M.}

\bigskip \textsc{\today}
\end{center}

\bigskip

\noindent\textbf{Abstract.}
We provide a functional characterization of isometries between non-reversible Finsler manifolds, in the form of a generalization of the Myers-Nakai Theorem for Riemannian manifolds. We show that, since non-reversible Finsler manifolds are a fundamentally asymmetric object, such a result can not be obtained by means of a symmetric function space, and we define the appropriate asymmetric structure needed to describe all possible isometries between this class of manifolds. The result is based on the ideas used in a previous generalization for reversible Finsler manifolds proved in \cite{GJR-13}, in which the normed algebra of $C^1$-smooth Lipschitz functions is used. To reflect the quasi-metric structure of non-reversible Finsler manifolds, this normed algebra had to be modified to include the cone of smooth semi-Lipschitz functions, resulting in a partial loss of the normed space and algebra structures. In order to achieve the desired result, we define new algebraic/quasi-metric structures to model the behavior of the aforementioned function space.
\bigskip


\vspace{0.6cm}



\tableofcontents

\section{Introduction}
The classical Myers-Nakai theorem asserts that the Riemannian structure of a Riemannian manifold $M$ is completely determined by the normed algebra structure of the associated space of real-valued, bounded functions of class $C^1$ with bounded derivatives, denoted by $C_b^1(M)$, when endowed with its natural norm $\|f\|=\max\{\|f\|_\infty,\|df\|_\infty\}$. More precisely, the theorem can be stated as follows.

\begin{theorem}[Myers-Nakai Theorem for Riemannian manifolds]\label{myers-nakairiemann}

Let $M$ and $N$ be two connected Riemannian manifolds. Then, $M$ and $N$ are isometric as Riemannian manifolds if and only if the spaces $C_b^1(M)$ and $C_b^1(N)$ are isometrically isomorphic as normed algebras. Moreover, every isometric isomorphism of normed algebras $T:C_b^1(N)\to C_b^1(M)$ must be of the form $Tf=f\circ \tau$, for a Riemannian isometry $\tau:M\to N$.
    
\end{theorem}
This result was proved by Myers in \cite{myers1954algebras} for compact manifolds and later by Nakai in \cite{nakai1959algebras} for the general case, and it is closely related to a previous result of Myers and Steenrod (see \cite{MS}), in which it was shown that, between connected Riemannian manifolds, Riemannian isometries coincide with metric isometries when viewing each manifold as a metric space. This result was generalized to Finsler manifolds by Deng and Hou (\cite{DH} and Theorem~\ref{denghou} bellow). 

The Myers-Nakai theorem can be regarded as a \emph{Banach-Stone type theorem}, as the classical Banach-Stone theorem states that compact and Hausdorff topological spaces are completely determined by the Banach space structure of the corresponding spaces of real-valued continuous functions. The term \emph{Banach-Stone type theorem} is commonly used to describe results that assert that a certain structure on a set $X$ is determined by the structure of a suitable space of real-valued functions over $X$. For more examples on Banach-Stone type theorems, we refer the reader to the survey \cite{garrido2002variations}. When viewed as a Banach-Stone type result, the Myers-Nakai theorem establishes an equivalence between \emph{algebra isometries} over the spaces $C_b^1(N)$ and $C_b^1(M)$ with \emph{isometries} between the Riemannian manifolds $M$ and $N$. The richness of the structures involved in this result opens up several avenues of possible extensions. One of these avenues was to extend the result to infinite dimensional Riemannian manifolds, which was done in \cite{GJR-09}. Two other extensions explored in the literature were to replace the equivalence notions between the functions spaces and the underlying manifolds, by normed algebra \emph{isomorphisms} and bi-Lipschitz diffeomorphisms respectively, and to extend the result to a wider class of smooth manifolds. Both results were accomplished in \cite{GJR-13} and \cite{GJR-09} respectively, where the authors focused on special classes of Finsler manifolds, obtaining the following notable results.

\begin{theorem}[Myers-Nakai for quasi-reversible Finsler manifolds, isomorphic version]\label{mnquasi}
  Let $M$ and $N$ be second countable, connected, (forward) complete and quasi-reversible Finsler manifolds. Then, there exists a bi-Lipschitz diffeomorphism between $M$ and $N$ if and only if the spaces $C_b^1(M)$ and $C_b^1(N)$ are isomorphic as normed algebras. Moreover, every normed algebra isomorphism ${T:C_b^1(N)\to C_b^1(M)}$ must be of the form $Tf=f\circ h$, for a bi-Lipschitz diffeomorphism ${h:M\to N}$.  
\end{theorem}

An isometric version of the result was obtained, provided the Finsler manifolds $M$ and $N$ are \emph{reversible}.

\begin{theorem}[Myers-Nakai for reversible Finsler manifolds, isometric version]\label{mnreversible}

Let $M$ and $N$ be second countable, connected, complete and reversible Finsler manifolds. Then, $M$ and $N$ are isometric as Finsler manifolds if and only if $C_b^1(M)$ and $C_b^1(N)$ are isometrically isomorphic as normed algebras. Moreover, every isometric isomorphism of normed algebras ${T:C_b^1(N)\to C_b^1(M)}$ must be of the form $Tf=f\circ h$, for a Finsler isometry $h:M\to N$.
\end{theorem}

These results were possible due to the fact that, for quasi-reversible Finsler manifolds (which include reversible Finsler manifolds and Riemannian manifolds), real-valued smooth functions have bounded derivative if and only if they are Lipschitz, which implies the algebra $C_b^1(M)$ can be seen as a subspace of $\mathrm{LIP}(M)$ (the space of real-valued Lipschitz functions over $M$), endowed with the norm  $\|f\|=\max\{\|f\|_\infty ,\|f\|_{\mathrm{LIP}}\}$, where $\|f\|_{\mathrm{LIP}}$ denotes the Lipschitz constant of the function $f$. 

The aim of this work is to provide a further generalization of both Theorems~\ref{mnquasi} and \ref{mnreversible} to the case of general (possibly asymmetric) Finsler manifolds. This poses a substantial challenge because, as we shall see in forthcoming Example~\ref{counterexample}, the normed algebra $C_b^1(M)$ is not well suited to describe isometries between non-reversible Finsler manifolds. The question then becomes: \emph{can we find a suitable algebraic structure and a function space $\mathcal{A}(\mathcal{X})$ that characterizes the structure of a general Finsler manifold $\mathcal{X}$?} In other words, we wish to define $\mathcal{A}(\mathcal{X})$, with a structure allowing for a result of the following type.
\begin{theorem}[Desired result]\label{desiredresult}
Let $(\mathcal{X},F_\mathcal{X})$ and $(\mathcal{Y},F_{\mathcal{Y}})$ be connected, second countable and forward complete Finsler manifolds. Then, $(\mathcal{X},F_\mathcal{X})$ and $(\mathcal{Y},F_{\mathcal{Y}})$ are Finsler isometric (respectively isomorphic) if and only if $\mathcal{A}(\mathcal{X})$ and $\mathcal{A}(\mathcal{Y})$ are isometrically isomorphic (respectively isomorphic). Moreover, every isomorphism (respectively isometry) between $\mathcal{A}(\mathcal{Y})$ and $\mathcal{A}(\mathcal{X})$ is of the form $Tf=f\circ h$, where $h:\mathcal{X}\to \mathcal{Y}$ is a Finsler isomorphism (respectively isometry).
\end{theorem}
In order to accomplish this, we will approach Finsler manifolds from the perspective of \emph{quasi-metric spaces}, that is, sets endowed with asymmetric distances. We shall see in Section~\ref{secprelim} that an appropriate notion of morphism between Finsler manifolds are smooth \emph{semi-Lipschitz} functions, while giving meaning to the phrase ``$\mathcal{A}(\mathcal{X})$ and $\mathcal{A}(\mathcal{Y})$ are isometrically isomorphic'' (or just isomorphic) is part of the task at hand. Proposition~\ref{normaderivadaslip} reinforces the choice of the quasi-metric framework, as it allows us to see the space of $C^1$-smooth functions with bounded derivative (with respect to Definition~\ref{asymnormder}) as a subset of $\mathrm{SLIP}(\mathcal{X})$ (the set of real-valued \emph{semi-Lipschitz} functions over $\mathcal{X}$), for any connected Finsler manifold $\mathcal{X}$. This observation provides a good starting point to search for the desired function space $\mathcal{A}(\mathcal{X})$, but at the same time points to an imminent obstacle, as the set $\mathrm{SLIP}(\mathcal{X})$ often fails to be a linear space, which is a consequence of the inherent asymmetric nature of non-reversible Finsler manifolds.  In order to work within the algebraic constraints of $\mathrm{SLIP}(\mathcal{X})$, new asymmetric structures had to be defined, which are presented in Section~\ref{secalg}. Due to the fact that the setting of quasi-metric spaces is not of common knowledge, we present in Section~\ref{secprelim} a brief review of the subject, containing the needed preliminary definitions and results.
\section{Preliminaries}\label{secprelim}
\subsection{Finsler manifolds}\label{prefinsler}
Finsler manifolds can be regarded as a generalization of Riemannian manifolds in which the tangent spaces are allowed to have a weaker notion of norm.

Let $E$ be a real vector space of finite dimension. A function $F:E\to [0,+\infty)$ is called a \emph{Minkowski norm} on $E$ if it satisfies:
		\begin{enumerate}
			\renewcommand\labelenumi{(\roman{enumi})}
			\item  $F(\lambda v)=\lambda F(v)$ for every $v\in E$ and $\lambda \geq 0$. 
			\item $F$ is continuous on $E$ and $C^\infty$-smooth on $E\setminus \{0\}$.
			\item For each $v\in E\setminus \{0\}$, the quadratic form  defined by $g_v=\tfrac{1}{2}d^2[F^2](v)$ is positive definite on $E$.
		\end{enumerate}
Additionally, every Minkowski norm satisfies the following conditions (see \cite[Theorem~1.2.2]{BCS} \emph{e.g.}):
\begin{enumerate}
 			\item[(iv)] $F(v)=0$ if and only if $v=0$. 
			\item[(v)] $F(u+v)\leq F(u)+F(v)$, for all $u,v\in E$.
\end{enumerate}
In general, Minkowski norms do not need to be symmetric. There are very interesting examples of spaces that utilize asymmetric Minkowski norms, such as Randers spaces (see \cite{BCS}) and Finsler manifolds. 

	\begin{definition}[Finsler manifold]\label{def-Finsler mnfd}
		A Finsler manifold is a pair $(\mathcal{X},F)$, where $\mathcal{X}$ is a finite-dimensional $C^\infty$-smooth manifold and $F:T\mathcal{X}\to [0,\infty)$ is a continuous function defined on the tangent bundle $T\mathcal{X}$, satisfying 
		\begin{enumerate}
			\renewcommand\labelenumi{(\roman{enumi})}
			\item $F$ is $C^\infty$-smooth on $T\mathcal{X}\setminus\{0\}.$ 
			\item For every $x\in \mathcal{X}$, the function $F(x,\cdot):T_x\mathcal{X}\to [0,\infty)$ is a Minkowski norm on the tangent space $T_x\mathcal{X}$.
		\end{enumerate}
	\end{definition}

		The function $F$, often called the \emph{Finsler structure}, is said to be \emph{reversible} if  $F(x,\cdot)$ is symmetric for every $x\in \mathcal{X}$. This notion generalizes Riemannian manifolds, as every such manifold is a reversible Finsler manifold, where the Minkowski norm on each tangent space is induced by an inner product. 
  \begin{definition}
		Let $(\mathcal{X},F)$ be a connected Finsler manifold. The \emph{Finsler distance} $d_F$ on $\mathcal{X}$ is defined by
		$$d_F(x,y)=\inf\{\ell_F(\sigma)\,:\,\sigma \text{ is a piecewise }C^1\text{ path from }x\text{ to }y\},$$
		where the Finsler length of a piecewise $C^1$ path $\sigma:[a,b]\to \mathcal{X}$ is defined as:
		$$\ell_F(\sigma)=\int_a^bF(\sigma(t),\dot{\sigma}(t))dt.$$
	\end{definition}
	
\begin{definition}[Finsler isometry]
	A mapping $\tau:(\mathcal{X},F)\to(\mathcal{Y},G)$ between Finsler manifolds is said to be a \emph{Finsler isometry} if it is a diffeomorphism which preserves the Finsler structure, that is, for every $x\in \mathcal{X}$ and every $v\in T_x\mathcal{X}$:
	$$F(x,v)=G(\tau(x),d\tau(x)(v)).$$
\end{definition}
We recall the following theorem of Deng and Hou \cite{DH}, which extends the result of Myers and Steenrod \cite{MS} from the Riemannian setting to the context of Finsler manifolds:

\begin{theorem}\label{denghou}
	Let $(\mathcal{X},F)$ and $(\mathcal{Y},G)$ be connected Finsler manifolds, and let ${\tau:(\mathcal{X},F)\to(\mathcal{Y},G)}$ be a bijection. Then, $\tau$ is a Finsler isometry if and only if it is an isometry for the corresponding Finsler distances, that is, for any $p,q\in \mathcal{X}$, we have
 $$d_F(p,q)=d_G(\tau(p),\tau(q)).$$
\end{theorem}
 Theorem \ref{denghou} suggests that, if one wishes to understand the structure of connected Finsler manifolds, there is much to be gained by studying them as \emph{quasi-metric spaces}.
\subsection{Quasi-metric spaces}\label{quasimetricspaces}
\begin{definition}[Quasi-metric space]\label{defquasimetric}
A \emph{quasi-metric space} is a pair $(X,d)$, where $X\neq\emptyset$ and $${d:X\times X\to [0,\infty)}$$
is a function, called {\em quasi-metric}, satisfying: 
\begin{enumerate}
\renewcommand\labelenumi{(\roman{enumi})}
\leftskip .35pc
     \item $\forall x,y,z\in X$: $d(x,y)\leq d(x,z)+d(z,y)$ \emph{(triangular inequality)};
     \item\label{separation} $\forall x,y\in X$: $x=y$ $\iff$ $d(x,y)=0$.
\end{enumerate}
\end{definition}
\noindent Unlike metrics, quasi-metrics are not required to be symmetric, that is, we may have that ${d(x,y)\neq d(y,x)}$ for some $x,y\in X$. If we replace condition~$\rm{(ii)}$ with
$$
{\rm (ii)' }\qquad x=y\quad\iff\quad \begin{cases}
\ d(x,y)=0\\[0.15cm]
\ d(y,x)=0
\end{cases}
$$
then we say that $d$ is a {\em quasi-hemi-metric}, in which case the distance between two different points is allowed to be $0$, as long as the distance in the ``opposite orientation" is strictly positive. The terminology of quasi-metric space will refer to a pair $(X,d)$ where $d$ is either a quasi-metric or a quasi-hemi-metric. In this work we shall also consider \emph{extended quasi-hemi-metrics} ${\tilde{d}:X\times X\to [0,+\infty]}$, that is, distance functions satisfying conditions $\rm (i)$ and $\rm (ii)'$, but which are also allowed to take the value $+\infty$. It is a well known fact that the Finsler distance $d_F$ of a connected Finsler manifold is always a quasi-metric (see \emph{e.g.} \cite[Section~6.2]{BCS}).\\

The reader should be alerted that the term {\em quasi-metric} is not standard. This term has been used by other authors to refer to other types of generalized distances. For example, in \cite{paluszynski2009quasi} and \cite{macias1979lipschitz}, the authors consider a relaxation of the triangular inequality while still requiring symmetry, while in \cite{alvarado2015hardy}, both the triangular inequality and symmetry are relaxed by constant factors. In this work, we will limit ourselves to Definition~\ref{defquasimetric}, as it is better suited to model the behavior of non-reversible Finsler manifolds.\\

Every  quasi-metric space $(X,d)$ can be endowed with a topology, known as the \emph{forward topology}, denoted by $\mathcal{T}(d)$,  generated by the family of open \emph{forward}-balls $\{B_{d}(x,r)\hbox{\rm :}\ x\in X, r>0\}$, where
$$B_{d}(x,r):=\{y\in X\hbox{\rm :}\ d(x,y)<r\}.$$

It is worth mentioning that other topologies can be defined on $X$ using the quasi-metrics \\${\bar{d}(x,y)=d(y,x)}$ and ${d^s(x,y)=\max\{d(x,y),d(y,x)\}}$ (which is a metric in the usual sense), and the same holds for extended quasi-metric spaces. For more information regarding these topologies, we refer the reader to \cite[Section 2]{daniilidis2021asymmetric}.

\begin{remark}[Topology of a Finsler manifold]
	For a connected Finsler manifold $(\mathcal{X},F)$, the forward topology associated with its Finsler distance $d_F$ coincides with the manifold topology (see \cite[Chapter~6.2]{BCS}). In particular, the forward topology of a connected Finsler manifold is always Hausdorff. 
\end{remark}

\begin{definition}[Forward completeness]\label{forwardcomp}
A sequence $(x_n)$ in a quasi-metric space $(X,d)$ is said to be \emph{forward-Cauchy} if for every $\varepsilon>0$, there exists $n_0\in \mathbb{N}$ such that $d(x_n,x_m)<\varepsilon$ whenever $n_0\leq n\leq m$. A space where every forward-Cauchy sequence is forward-convergent is called \emph{forward complete}.
\end{definition}

As it is the case in the classical theory of metric spaces, the special case when the distance function is associated with a norm deserves to be addressed on its own right. 
\begin{definition}[Asymmetric normed space]\label{defasymmnorm}
For a real linear space $E$, an {\em asymmetric hemi-norm} is a function $\|\cdot|:E\to\mathbb{R}_+$ that verifies triangular inequality, positive homogeneity and the following separation property:
\begin{equation*}\label{eq:Separation-heminorm}
{\rm (ii)' }\quad x=0\quad\iff\quad \begin{cases}\
\|x|=0\\[0.15cm]
\ \|\!-\!\!x|=0
\end{cases}
\end{equation*}
A vector space endowed with an asymmetric hemi-norm will be called an \emph{asymmetric normed space}.
\end{definition}
The symbol $\|\cdot\vert$, using two vertical bars on the left and only one bar on the right side, serves as a reminder of the asymmetric nature of these type of functions, in the sense that the values $\|x\vert$ and $\| -x\vert$ may not coincide. We may also consider, keeping the same notation, {\it extended asymmetric hemi-norms}, allowing $\|\cdot|$ to take the value $+\infty$. Clearly, any extended asymmetric normed space can be endowed with the extended quasi-hemi-metric 
\begin{equation}\label{qnqd}
d(x,y):=\|y-x|,\quad\text{for all } x,y\in X.
\end{equation}

 The notion of asymmetric norm can be generalized further by allowing asymmetries in an algebraic sense, which will be done by considering cones instead of linear spaces.

\begin{definition}[Cone] \label{defcone}\em
 A subset $C$ of a real linear space $E$ will be called a \emph{cone} if it is closed under finite sums and under multiplication by non-negative scalars. In other words, for all $x,y\in C$ and $r\geq 0$, we have $ x+y\in C$ and $r\cdot x\in C$.

In particular, under this definition, every cone is a convex subset of $E$ containing the origin. A \emph{subcone} of a cone $C$ will be any set that is a cone itself, and a \emph{linear map} on a cone $C$ will be the restriction of a linear map (with values on some linear space $F$) on the linear space $\mathrm{span}(C)=C-C\subseteq E$.
\end{definition}

We will consider cones endowed with an asymmetric norm, as follows.

\begin{definition}[Conic norm] \label{defnorm}\em
A {\em conic hemi-norm} on a cone $C$ is a function $\|\cdot|\hbox{\rm :}\ C\to \mathbb{R}_{+}$ such that for all
$x,y\in C$ and $r\geq 0$:\smallskip
\begin{enumerate}
\renewcommand\labelenumi{(\roman{enumi})}
\leftskip .35pc
\item $\|x+y|\leq \|x|+\|y|$,  for all $x,y \in C$ \smallskip
\item $ x=0\iff -x\in C$ and $\|-x|=0=\|x|$ \smallskip
\item $\|r\cdot x|=r\|x|$,  for all $x\in C$ and $r \geq 0$. \vspace{-.4pc}\smallskip
\end{enumerate}
\medskip
The pair $(C,\|\cdot|)$ is called \emph{normed cone}.
\end{definition}
\begin{remark}\em
Asymmetric normed spaces (as per Definition \ref{defasymmnorm}) are a particular case of normed cones.
\end{remark}

\subsection{Semi-Lipschitz functions}\label{subsecslip}
We start this subsection with a brief discussion on the importance of the quasi-metric space $(\mathbb{R},d_u)$, which will play a key role in the definition of real valued semi-Lipschitz functions. Consider the asymmetric hemi-norm $u(x)=\max\{x,0\}$ and the associated quasi-hemi-metric $d_u$ is given by 
\begin{equation}\label{du}
d_u(x,y)=u(y-x)=\max\{y-x,0\}\quad\text{for every } x,y\in\mathbb{R}.
\end{equation}

The reader should mind the order of the variables $x$ and $y$ on~\eqref{du}. Since $u$ is an asymmetric hemi-norm, the order in which we write the difference between $y$ and $x$ (as opposed to $x-y$) does matter. We chose this order because it aligns with the notion that, in a (symmetric) normed linear space, the norm of a vector $v$ represents the length of an arrow going from the origin to the position of $v$, that is, $d(0,v)=\|v-0\|$.

The forward topology of this quasi-metric space characterizes the usual notion of upper semicontinuity for real-valued functions (see \cite{daniilidis2021asymmetric}), while the \textit{backward} topology, defined using $\bar{d_u}(x,y)=d_u(y,x)$, holds the same property for lower semicontinuous functions. Therefore, in order to define real valued semi-Lipschitz functions, we must make a choice regarding the ``orientation'' of the quasi-metric on $\mathbb{R}$, which will impact whether the resulting functions are upper or lower semicontinuous, as well as the computation of the \textit{semi-Lipschitz constants}. As for the general definition of semi-Lipschitz functions, it comes as a straightforward generalization of the usual definition of Lipschitz functions between metric spaces. 
\begin{definition}\label{sLipfunction}
Let $(X,d)$ and $(Y,\rho)$ be quasi-metric spaces. A function $f:X\to Y$ is said to be \emph{semi-Lipschitz} if there exists a constant $L\geq 0$ such that for every $x,y\in X$ we have
\begin{equation}\label{sLipineq}
\rho(f(x),f(y))\leq Ld(x,y).
\end{equation}
\end{definition}
The infimum of all constants satisfying \eqref{sLipineq} is called the \emph{semi-Lipschitz constant} of $f$, and is denoted by $\mathrm{SLIP}_{d,\rho}(f)$, or simply $\mathrm{SLIP}(f)$ when there is no risk of confusion. That is,
$$\mathrm{SLIP}(f):=\inf\big\{L>0: \,\,\eqref{sLipineq} \text{ holds}\big\}.$$
Moreover, if $f:X\to Y$ is a semi-Lipschitz function between quasi-metric spaces $(X,d)$ and $(Y,\rho)$, then,
$$\mathrm{SLIP}(f)=\sup_{d(x,y)>0}\frac{\rho(f(x),f(y))}{d(x,y)}.$$
In addition, we say that a function $g$ is $\lambda$-semi-Lipschitz whenever $\mathrm{SLIP}(g)\leq \lambda$.
Notice that, unlike in metric spaces, the conditions ``$d(x,y)>0$'' and ``$x\neq y$" are not equivalent. Nevertheless, this two conditions coincide whenever $d$ is a quasi-metric (as opposed to a quasi-hemi-metric), which is always the case for connected Finsler manifolds. 

In this work, for a real-valued function $f:(X,d)\to\mathbb{R}$, we say that $f$ is semi-Lipschitz if it so for the quasi-metric $d_u$. That is, if there exists $L\geq 0$ satisfying 
$$ d_u(f(x),f(y))\leq L d(x,y),$$
or equivalently, 
\begin{equation}\label{slipinequality}
  f(y)-f(x)\leq L d(x,y),
\end{equation}
\noindent where the conic hemi-norm $u(\cdot)$ is omitted on the left side of~\eqref{slipinequality}, as $t= u(t)$ for any $t\in \mathbb{R}_+$ and the inequality is trivially satisfied whenever the left side is negative. Note that one could consider semi-Lipschitzianity with respect to the reverse quasi-hemi-metric ${\bar{d_u}(x,y)=d_u(y,x)}$. This alternative has been implicitly used in the literature. However, we prefer to consider semi-Lipschitzianity as in \eqref{slipinequality}, since it is better adapted to asymmetric norms. We refer the reader to \cite[Remark~2.31]{daniilidis2021asymmetric} for a discussion on the subject.

For any quasi-metric space $(X,d)$, we write $\mathrm{SLIP}(X)$ to denote the space of real-valued semi-Lipschitz functions, where $\mathbb{R}$ is endowed with its usual quasi-metric $d_u$. When $(X,d)$ is a \emph{pointed} metric space (that is, a metric space with a distinguished point $x_0$), $\mathrm{SLIP}_0(X)$ will denote the subspace of semi-Lipschitz functions mapping the base point $x_0$ to $0$. A normed cone will always be considered as having its origin as its base point.

From this point onward, we will use the symbol $\|\cdot\vert_S$ to denote the semi-Lipschitz constant of a real-valued function. The reason for this notation is the fact that, for any pointed quasi-metric space $(X,d)$, $(\mathrm{SLIP}_0(X),\|\cdot\vert_S)$ is a normed cone (with respect to Definition~\ref{defnorm}).

\begin{remark}\label{Slipnnotlinear}
 In general, $\mathrm{SLIP}(X)$ and $\mathrm{SLIP}_0(X)$) need not be linear spaces. This can be seen by taking $(X,d)=(\mathbb{R},d_u)$, and considering the function ${f(x)=d_u(0,x)}$, which belongs to $\mathrm{SLIP}_0(X)$, while $-f$ does not. A more nuanced example will be given after Proposition~\ref{nonlinearbachirflores}.
\end{remark}

The following proposition ensures continuity of semi-Lipschitz functions. The proof is omitted, as it follows immediately from the definitions involved.
\begin{proposition}\label{contslip}
    Let ${f:(X,d)\to (Y,\rho)}$ be a semi-Lipschitz function between quasi-metric spaces $(X,d)$ and $(Y,\rho)$. Then, $f$ is continuous for the respective forward topologies of $(X,d)$ and $(Y,\rho)$. In particular, real-valued semi-Lipschitz functions are upper-semicontinuous.
\end{proposition}

\begin{proposition}\label{dist}
    Let $(X,d)$ be a quasi-metric space with at least two points. Then, for every $x\in X$, the function ${d(x,\cdot):(X,d)\to (\mathbb{R},d_u)}$ is semi-Lipschitz and has semi-Lipschitz constant equal to $1$.
\end{proposition}

Next, we present some definitions and results concerning semi-Lipschitz functions on Finsler manifolds. 

\begin{definition}[Asymmetric norm of a derivative]\label{asymnormder}
Let $(\mathcal{X},F)$ be a connected Finsler manifold and $f:\mathcal{X}\to \mathbb{R}$ a $C^1$-smooth function. The asymmetric norm of the derivative of $f$ at the point $x\in \mathcal{X}$ is defined by:
$$\|df(x)|_F=\sup\{df(x)(v)\::\:v\in T_x\mathcal{X},\:F(x,v)=1\}.$$
\end{definition}

The following proposition is a straightforward extension of \cite[Theorem 5]{GJR-13}, which relates the \emph{Lipschitz} constant of a $C^1$-smooth function with the supremum norm of its derivative. The same proof given in \cite{GJR-13} yields the corresponding one-sided result:

\begin{proposition}\label{normaderivadaslip}
Let $(\mathcal{X},F)$ be a connected Finsler manifold and $f:\mathcal{X}\to \mathbb{R}$ a $C^1$-smooth function. Then 
$$ \|f|_{S}=\|df|_{\infty}:=\sup\{\|df(x)|_F\::\:x\in \mathcal{X}\}\in [0,\infty],$$
where $\|f|_{S}$ is the semi-Lipschitz constant of f with respect to the Finsler quasi-metric $d_{F}$.
\end{proposition}

We now present an asymmetric version of an approximation result from \cite[Theorem 8]{GJR-13}, which was shown in \cite{DJV}. This will be one of the key tools for the proof of our main result.

\begin{corollary}[Smooth approximation of semi-Lipschitz functions]\label{smoothslip}
Let $(\mathcal{X},F)$ be a connected, second countable Finsler manifold, $f:\mathcal{X}\to \mathbb{R}$ a semi-Lipschitz function, $\varepsilon:\mathcal{X}\to (0,+\infty)$ a continuous function and $r>0$. Then, there exists a $C^1$-smooth semi-Lipschitz function $g:\mathcal{X}\to \mathbb{R}$ such that:
\begin{enumerate}
\leftskip .35pc
    \item[$(i)$] $|g(x)-f(x)|\leq \varepsilon(x)$ for all $x\in \mathcal{X}$ ; 
    \item[$(ii)$] $\|g|_{S}\leq \|f|_{S}+r$.
\end{enumerate}
\end{corollary}
The proof of Corollary~\ref{smoothslip} (which is based in results analogous to Proposition 6 and Lemma~7 of \cite{GJR-13}) is omitted, since all arguments are straightforward adaptations of the aforementioned ones, replacing Lipschitz bounds with semi-Lipschitz ones. Nevertheless, we state the asymmetric version of Proposition 6 of \cite{GJR-13}, as it will be of use later. 

\begin{proposition}\label{locslip}
A function ${f:\mathcal{X}\to \mathcal{Y}}$ between connected Finsler manifolds $\mathcal{X}$ and $\mathcal{Y}$ is {$C$-semi-Lipschitz} if and only if it is \emph{locally} $C$-semi-Lipschitz, that is, for every $x\in \mathcal{X}$ there exists a neighborhood $U$ of $x$ such that $f|_U$ is $C$-semi-Lipschitz.
\end{proposition}

\section{Algebraic challenges of non-reversibility}\label{secalg}

As already mentioned in Remark~\ref{Slipnnotlinear}, the set of real-valued semi-Lipschitz functions over a quasi-metric space may fail to be a linear space. This peculiarity makes the study of these types of functions a more challenging endeavor, as we can not directly apply the usual tools and results of functional analysis. This phenomenon is intimately related to the question of how asymmetric a given quasi-metric space $(X,d)$ is, in the sense of the relation (or lack thereof) between the quantities $d(x,y)$ and $d(y,x)$. A notion to quantify asymmetry was introduced independently by Rademacher in \cite{rademacher2004sphere} and by Bachir and Flores in \cite{bachir10.1216/rmj.2020.50.1951}. For a quasi-metric space $(X,d)$, the \emph{index of symmetry} of $(X,d)$ is defined by
\begin{equation*}
   c(X,d)=\inf_{d(x,y)>0}\quad\frac{d(y,x)}{d(x,y)} \in [0,1]. 
\end{equation*}
\noindent If $(X,d)$ is an extended quasi-metric space with at least two points at infinite distance, the index of symmetry will be considered to be $0$.

Clearly, the class of quasi-metric spaces with index of symmetry $1$ is exactly the class of metric spaces. It is also easy to check that $c(\mathbb{R},d_u)=0$. On the other hand, if $(\mathcal{X},d_\mathcal{X})$ is a connected Finsler manifold, then every point $p\in \mathcal{X}$ has a neighborhood $U_p$ such that the index of symmetry of the subspace $(U_p,d_\mathcal{X}|_{U_p})$ is strictly positive (this can be proved using the fact that Minkowski norms are ``equivalent'' to the Euclidean norm, see \cite[Lemma 6.2.1]{BCS}). It follows that the index of symmetry of every compact and connected Finsler manifold is strictly positive. Finsler manifolds with positive index of symmetry are called \emph{quasi-reversible} in \cite{GJR-13}.

A simple estimation shows that whenever there is a bi-semi-Lipschitz mapping between quasi-metric spaces $(X,d)$ and $(Y,\rho)$ (that is, a semi-Lipschitz bijection with semi-Lipschitz inverse), the \emph{sign} of $c(X,d)$ and $c(Y,\rho)$ must coincide, in the sense that if one index is $0$, the other one must be $0$ as well.

\begin{example}[Finsler manifold with index of symmetry $0$]\label{finslerindex0}
 Consider $\mathbb{R}$, endowed with the Finsler structure $F(x, v) = \vert v \vert - d \phi(x)(v)$, where $\phi : \mathbb R \to \mathbb{R}$ is given by
$$
\phi (x) := \int_0^x \frac{t^2}{1+t^2} \, dt.
$$
It is easy to see that the associated Finsler distance is $d_{F} (x, x')= \vert x-x'\vert + \phi (x) - \phi(x')$. To check that $c(\mathbb{R},d_F)=0$, consider points $x$ and $x'=x+1$. Using the definitions of $\phi$ and $d_F$, we can compute
$$ c(\mathbb{R},d_F)\leq \frac{d_F(x,x+1)}{d_F(x+1,x)}=\frac{\arctan(x+1)-\arctan(x)}{2+\arctan(x)-\arctan(x+1)},$$
\noindent which converges to $0$ when $x\to +\infty$.
\end{example}

It was shown in \cite{Tesis} that the index of symmetry can be used to characterize the class of quasi-metric spaces $(X,d)$ such that $\mathrm{SLIP}_0(X)$ is a linear space.

\begin{proposition}\label{nonlinearbachirflores}
Let $(X,d)$ be a quasi-metric space. Then, the following are equivalent:
\begin{itemize}
    \item[$(i)$] $c(X,d)>0$
    \item[$(ii)$] $\mathrm{SLIP}_0(X)$ is a linear space.
\end{itemize}
\end{proposition}

In particular, we have that for the Finsler manifold $\mathcal{X}=(\mathbb{R},F)$ of Example~\ref{finslerindex0}, $\mathrm{SLIP}_0(\mathcal{X})$ is not a linear space. The fact that this can happen even for quasi-metric spaces as well behaved as Finsler manifolds points to the necessity of working within an asymmetric framework for the result we aim to prove. To further illustrate the necessity of replacing the normed algebra $C_b^1(M)$ used in Theorems~\ref{myers-nakairiemann}, \ref{mnquasi} and \ref{mnreversible} by an asymmetric object, consider the following example. 

\begin{example}\label{counterexample}
Let $(\mathcal{X},F)$ be the non-reversible Finsler manifold defined in Example~\ref{finslerindex0}, and consider the normed algebras $C_b^1(\mathbb{R},|\cdot|)$ and $C_b^1(\mathcal{X})$. The norm of $C_b^1(\mathbb{R},|\cdot|)$ is given by ${\|f\| = \max\{\|f\|_{\infty},\|df\|_{\infty}\}}$, where  $\|df\|_{\infty}$ coincides with the usual Lipschitz semi-norm. For $C_b^1(\mathcal{X})$, a similar norm is considered, taking the maximum between  $\|f\|_{\infty}$  and the quantity ${\|df\|_{\infty}^F = \sup\{ |df(x)(v)|\,:\, (x,v)\in T\mathcal{X}, F(x,v)\leq 1\}}$ (see \cite{GJR-13} for the details). Using the definition of $F$, that is, $F(x,v)=|v|-d\phi(x)(v)$, it is easy to show that for any $f\in C_b^1(\mathbb{R})$, we have that
$$ \|df\|_\infty\leq \|df\|^F_\infty \leq 2\|df\|_\infty.$$
This implies that the identity mapping $\mathrm{Id}: C_b^1(\mathbb{R},|\cdot|) \to C_b^1(\mathcal{X})$ is an isomorphism of normed algebras. Yet, there cannot exist a bi-semi-Lipschitz mapping $\tau:\mathcal{X}\to \mathbb{R}$, as the indexes of symmetry of the manifolds are $0$ and $1$, respectively.
\end{example}
This should be contrasted with Theorem~\ref{mnquasi}, which asserts that every normed algebra isomorphism between $C_b^1(M)$ and $C_b^1(N)$ induces a bi-Lipschitz mapping between the Finsler manifolds $N$ and $M$, provided both have positive index of symmetry. We have shown that such a result can not hold for general Finsler manifolds, which points to the limitations of working within the symmetric framework of normed algebras. A possible explanation of this shortcoming of the space $C_b^1(M)$ lies in the fact that, whenever $M$ is a quasi-reversible Finsler manifold, the functions of the form $d_F(m,\cdot)$ are Lipschitz for all $m\in M$, with Lipschitz constant bounded by $c(M,d_F)^{-1}$ (see \cite[Theorem 11]{GJR-13}). This functions, which are an essential component of the technique used in \cite{GJR-10,GJR-09,GJR-13}, can be truncated and smoothly approximated with functions within $C_b^1(M)$ (see \cite{GJR-13}), which is a crucial step of the proof of Theorem~\ref{mnquasi}. On the other hand, for a non-reversible Finsler manifold, the functions $d_F(m,\cdot)$ may fail to be Lipschitz (with respect to the usual metric on $\mathbb{R}$), but are semi-Lipschitz instead. They can still be truncated and smoothly approximated (using Corollary~\ref{smoothslip}), but these smooth approximations may not belong to $C_b^1(M)$. They will however belong to a suitable space of smooth semi-Lipschitz functions, which we will now define. 

\begin{definition}
Let $(\mathcal{X},F)$ be a second countable and connected Finsler manifold. We will consider the following sets of $C^1$-smooth, bounded, semi-Lipschitz and real-valued functions on $\mathcal{X}$:
\begin{align*}
SC_b^1(\mathcal{X}):=&\{f\in C^1(\mathcal{X})\,:\,\|f\|_\infty<\infty,\,\|f|_S<\infty\},\\
SC_b^1(\mathcal{X})_+:=&\{f\in SC_b^1(\mathcal{X})\,:\,f\geq 0\}.
\end{align*}
\end{definition}

Clearly, both sets are subcones of $\mathrm{SLIP}(\mathcal{X})$, and they can be endowed with the natural asymmetric hemi-norm $\|f|=\max\{\|f\|_\infty,\|f|_S\}.$ Note that, thanks to Proposition~\ref{normaderivadaslip}, an equivalent definition for both spaces and their conic hemi-norm can be given by asking for the differential of each function to be bounded (in the sense of Definition~\ref{asymnormder}). 

Given that, in general, the cone $SC_b^1(\mathcal{X})$ fails to be a linear space, we cannot expect $SC_b^1(\mathcal{X})$ to be an algebra in the usual sense. Moreover, the lack of additive inverses makes it impossible for $SC_b^1(\mathcal{X})$ to be closed under multiplication. Indeed, consider a function $f\in SC_b^1(\mathcal{X})$ such that $-f$ does not belong to $SC_b^1(\mathcal{X})$ (as in Example~\ref{finslerindex0}). Then, the product of the constant function of value $(-1)$ and $f$ does not belong to $SC_b^1(\mathcal{X})$.
Despite this negative observation, we can avoid this problem by restricting ourselves to functions with non negative values.
\begin{proposition}\label{sc+mult}
 Let $(\mathcal{X},F)$ be a Finsler manifold. The set $SC_b^1(\mathcal{X})_+$ is closed under pointwise multiplication. Moreover, for any $f,g\in SC_b^1(\mathcal{X})_+$, we have
\begin{equation}\label{submult}
    \|f g|\leq 2 \|f| \|g|.
\end{equation}
\end{proposition}
\begin{proof}
It is clear that \eqref{submult} implies that the function $fg$ belongs to $SC_b^1(\mathcal{X})_+$. Moreover, it is also clear that $\|fg\|_\infty \leq \|f\|_\infty\|g\|_\infty\leq \|f|\|g|$, so we only need to prove that $\|fg|_S\leq 2\|f|\|g|$. To this end, let $x,y\in \mathcal{X}$, and recall the notation $u(t)=\max\{t,0\}$ for $t\in \mathbb{R}$. Since $f$ and $g$ are non negative, we have
\begin{align*}
fg(x)- fg(y)&= f(x)\left( g(x)-g(y)\right)+g(y)\left( f(x)-f(y)\right)\\
&\leq f(x)u\left( g(x)-g(y)\right)+g(y)u\left( f(x)-f(y)\right)\\
&\leq \|f\|_\infty u\left( g(x)-g(y)\right) +  \|g\|_\infty u\left( f(x)-f(y)\right)\\
&\leq \|f\|_\infty \|g|_S d_{\mathcal{X}}(y,x) + \|g\|_\infty \|f|_S d_{\mathcal{X}}(y,x) \\
&= ( \|f\|_\infty \|g|_S+\|g\|_\infty \|f|_S)d_{\mathcal{X}}(y,x)\\
&\leq 2\max\{\|f\|_\infty \|g|_S, \|g\|_\infty \|f|_S\}d_{\mathcal{X}}(y,x)\\
&\leq 2d_{\mathcal{X}}(y,x)\max\Bigl\{\|g|_S\max\{\|f\|_\infty,\|f|_S\}, \|f|_S \max\{\|g\|_\infty, \|g|_S\}\Bigr\}\\
&=2 \max\{\|f|\|g|_S ,\|g|\|f|_S\}d_{\mathcal{X}}(y,x)\\
&\leq 2\|f|\|g|d_{\mathcal{X}}(y,x).
\end{align*}
The proof is then finished.
\end{proof}
Whit this, $SC_b^1(\mathcal{X})_+$ is:
\begin{itemize}
    \item[$(i)$] A convex cone closed under the operations of addition and scalar multiplication (using non negative scalars).
    \item[$(ii)$] Closed under the operation of pointwise multiplication, which distributes over addition and scalar multiplication.
    \item[$(iii)$] A normed cone satisfying $\|fg|\leq 2 \|f|\|g|$ for all ${f,g\in SC_b^1(\mathcal{X})_+}$.
\end{itemize}
Inspired by this properties, we introduce new definitions.
\subsection{New asymmetric structures}\label{newstructures}
First, we recall the algebraic definition of a semiring: a semiring is a commutative monoid (that is, a semigroup with a zero element) endowed with a compatible multiplication operation that distributes over the addition of the monoid. 

\begin{definition}[Conic-semiring]\label{conicsemiring}
 A conic-semiring will be a cone (as per Definition~\ref{defcone}) endowed with a multiplication that makes it a semiring. If the cone is endowed with a conic hemi-norm for which there exists a constant $K>0$ such that $\|fg|\leq K\|f|\|g|$ for all $f,g$ in the cone, we will call it a \emph{normed conic-semiring}. A normed conic-semiring will be called \emph{unital} if it has a multiplicative unit.
\end{definition}

Just like the notion of cones is the asymmetric version of real linear spaces (by using monoids instead of groups and $\mathbb{R}_+$ instead of $\mathbb{R}$ for scalars), the notion of semiring can be seen as the asymmetric version of rings (replacing the additive group with a monoid). By combining these two ideas, we can view conic-semirings as an asymmetric version of algebras. Considering that normed algebras have algebraic and metric components, we can see Definition~\ref{conicsemiring} as a generalization which forgoes the former (linear and ring structure) in order to preserve the latter (having a well defined norm). We will also explore the opposite idea, that is, sacrificing properties of the norm in order to preserve the algebraic structure. Recall that an \emph{extended} asymmetric hemi-norm has the same properties as an asymmetric hemi-norm (see Definition~\ref{defasymmnorm}), but is allowed to take the value $+\infty$. The notion of (symmetric) extended norms was studied by Beer and Vanderwerff in \cite{beer2015structural}, and was generalized shortly after by Salas and Tapia-Garc\'ia in \cite{salas2016extended}, where they studied extended seminormed and extended topological vector spaces. To the best of our knowledge, this idea has not yet been combined with asymmetric structures. The proposed relaxation of asymmetric normed spaces will allow us to study linear spaces where the asymmetric hemi-norm is not always well defined (in the sense of taking finite values).

\begin{definition}[Finite subcone]\label{finitesubcone}
Let $(E,\|\cdot|)$ be an extended asymmetric normed space. The subset $F=\{ x\in E ~:~ \|x|<+\infty\}$ (which is always a cone) is called the \emph{finite subcone} of $E$.    
\end{definition}

\begin{definition}[Extended asymmetric normed algebra]\label{extasymalg} An algebra $\mathcal{A}$ endowed with an extended asymmetric hemi-norm $\|\cdot |$ will be called an \emph{extended asymmetric normed algebra} if the finite subcone satisfies a sub-multiplicative condition for the norm, i.e., there exists $K>0$ such that $\|fg|\leq K\|f|\|g|$ for all $f$ and $g$ with finite norm.
\end{definition}

This new definition is clearly connected with normed conic-semirings, as the finite subcone of any extended asymmetric normed algebra is a normed conic-semiring. Conversely, one can also define an extended asymmetric normed algebra from a given normed conic-semiring. When working with a cone $C\subseteq\mathbb{R}^X$ for some set $X$, $\mathrm{span}(C)$ will denote the linear span of $C$ in $\mathbb{R}^X$.

\begin{proposition}\label{conicsemiringtoalgebra}
    Let $(C,\|\cdot|)$ be a normed conic-semiring contained in $\mathbb{R}^X$ for some set $X$. Set $\mathcal{A}=\mathrm{span}(C)$, and for any $a\in \mathcal{A}$, define
$$
\| a |_{\mathcal{A}}=\begin{cases}
\begin{array}{ccc}
&\| a| &\text{ if }a\in C \\
&+\infty&\text{ if }a\notin C
\end{array}
.\end{cases}
$$
Then, $(\mathcal{A}, \| \cdot |_{\mathcal{A}})$ is an extended asymmetric normed algebra, and the finite subcone of $\mathcal{A}$ coincides with $C$. In this case, we will say that the extended asymmetric normed algebra $\mathcal{A}$ is \emph{generated} by the normed conic-semiring $C$.
\end{proposition}
\begin{proof}
Let us verify that $\mathcal{A}$ is an algebra. Since it is by definition a linear space, we only need to check that it is closed under multiplication. Let $x,y\in \mathcal{A}$. Since $C$ is a cone, $\mathrm{span}(C)$ can be written as $C-C=\{c_1 - c_2 ~:~ c_1,c_2\in C\}$, so we can write 
$$xy= (x_1-x_2)(y_1-y_2)=x_1y_1-x_1y_2-x_2y_1+x_2y_2=(x_1y_1+x_2y_2) - (x_1y_2 +x_2y_1),$$
\noindent with $x_i$ and $y_i$ in $C$, for $i=\{1,2\}$. It follows that $xy\in C-C=\mathcal{A}$. The remaining properties hold by definition. 
\end{proof}
Using Proposition~\ref{conicsemiringtoalgebra}, we can define our desired function space.
\begin{definition}\label{algebrafinsler}
Let $(\mathcal{X},F)$ be a connected Finsler manifold. We define $\mathcal{A}(\mathcal{X})$ to be the extended asymmetric normed algebra generated by the conic-semiring $SC_b^1(\mathcal{X})_+$, that is,
$$\mathcal{A}(\mathcal{X})=\mathrm{span}\left( SC_b^1(\mathcal{X})_+\right).$$   
The extended asymmetric hemi-norm of an element $f\in \mathcal{A}(\mathcal{X})$ is given by $\|f|=\max\{\| f\|_\infty,\|f|_S\}$ if ${f\in SC_b^1(\mathcal{X})_+}$, and $+\infty$ otherwise.
\end{definition}
Notice that, since every function in $SC_b^1(\mathcal{X})$ is bounded, we can write $f$ as ${(f + \|f\|_\infty) - \|f\|_\infty}$, which belongs to $\mathcal{A}(\mathcal{X})$. Therefore, $SC_b^1(\mathcal{X})$ and $SC_b^1(\mathcal{X})_+$ have the same linear span (although they do not induce the same extended asymmetric hemi-norm on $\mathcal{A}(\mathcal{X})$). It follows that if $\mathcal{X}$ is a reversible Finsler manifold (in particular, if $\mathcal{X}$ is Riemannian), the normed algebras $\mathcal{A}(\mathcal{X})$ and $C_b^1(\mathcal{X})$ coincide.
\begin{definition}\label{continuityextnorm}
A linear function $T$ between two extended asymmetric normed algebras $(\mathcal{A}_2,\|\cdot|_2)$ and $(\mathcal{A}_1,\|\cdot|_1)$ is said to be \emph{forward continuous} (or just continuous) if there exists a constant $K\geq 0$ such that $$\|Tf|_1\leq K\|f|_2,$$
\noindent for all $f\in \mathcal{A}_2$. The least constant $K$ satisfying this inequality is called the \emph{norm} of $T$, denoted by $\|T|$.
\end{definition}

Note that a continuous linear function necessarily maps the finite subcone of its domain into the finite subcone of its range.

\begin{definition}\label{extasymiso}
Given two extended asymmetric normed algebras $(\mathcal{A}_1,\|\cdot|_1)$ and $(\mathcal{A}_2,\|\cdot|_2)$, a mapping $T:\mathcal{A}_2\to \mathcal{A}_1$ is called an \emph{extended asymmetric normed algebra isomorphism} provided:
\begin{enumerate}
    \item[$(i)$] $T$ is linear and bijective,
    \item[$(ii)$] $T$ is bicontinuous, i.e., $T$ and $T^{-1}$ are continuous in the sense of Definition~\ref{continuityextnorm}.
    \item[$(iii)$] $T(fg)=Tf\cdot Tg$ for all $f,g\in \mathcal{A}_2.$
\end{enumerate}
The isomorphism $T$ is called an extended asymmetric normed algebra \emph{isometry} if $\|Tf|_1=\|f|_2$ for all $f\in \mathcal{A}_2$, or equivalently, if $\|T|=\|T^{-1}|=1$.
\end{definition}

In what follows, whenever we mention the dual $\mathcal{A}^*$ of an extended asymmetric normed algebra $\mathcal{A}$, it will be its dual cone when viewing $\mathcal{A}$ as an extended asymmetric normed space, that is, 

$$ \mathcal{A}^*:=\left\{\varphi:\mathcal{A}\to \mathbb{R}~:~\varphi\text{ is linear and forward upper semicontinuous} \right\}.$$
We remark that the forward upper semicontinuity of the linear function $\varphi$ is equivalent to it being semi-Lipschitz, or that the quantity 
$$\|\varphi|^*:=\sup_{\|a|\leq 1}\varphi(a)$$
\noindent which will be referred as the \emph{dual norm} of the function $\varphi$, is finite (see \cite[Proposition 2.35]{daniilidis2021asymmetric}).
Given a functional  $\varphi \in \mathcal{A}^*$, we say $\varphi$ is \emph{multiplicative} if $\varphi(ab)=\varphi (a) \varphi(b)$ for all $a,b\in \mathcal{A}$. If the algebra $\mathcal{A}$ is contained in $\mathbb{R}^X$ for some set $X$ (i.e., $\mathcal{A}$ is an algebra of functions over $X$), we will say a functional $\varphi \in \mathcal{A}^*$ is \emph{positive} if $\varphi(f)\geq 0$ whenever $f\geq 0$ as a function.

\section{Proof of the main result}\label{secmn}

Our aim is to prove that for every pair of connected, second countable and forward complete Finsler manifolds $\mathcal{X}$ and $\mathcal{Y}$, every extended asymmetric normed algebra isometry between $\mathcal{A}(\mathcal{Y})$ and $\mathcal{A}(\mathcal{X})$ induces an isometry between $\mathcal{X}$ and $\mathcal{Y}$. More precisely, we will prove the following. 

\begin{theorem}[Non reversible Myers-Nakai Theorem, isometric version]\label{thm:main}
Let $(\mathcal{X},F_\mathcal{X})$ and $(\mathcal{Y},F_{\mathcal{Y}})$ be connected, second countable and forward complete Finsler manifolds. Then, $(\mathcal{X},F_\mathcal{X})$ and $(\mathcal{Y},F_{\mathcal{Y}})$ are Finsler isometric if and only if there exists an extended asymmetric normed algebra isometry ${T:\mathcal{A}(\mathcal{Y})\to \mathcal{A}(\mathcal{X})}$. Moreover, every extended asymmetric normed algebra isometry between $\mathcal{A}(\mathcal{Y})$ and $\mathcal{A}(\mathcal{X})$ is of the form $Tf=f\circ h$, where $h:\mathcal{X}\to \mathcal{Y}$ is a Finsler isometry.
\end{theorem}

To this end, we follow the scheme used in \cite{GJR-13}, with the necessary adjustments required for our asymmetric framework.
\begin{definition}
Let $(\mathcal{X},F)$ be a Finsler manifold. We define the \emph{structure space} of $\mathcal{X}$ as
$$ \mathcal{S}(\mathcal{X}):=\{\varphi:\mathcal{A}(\mathcal{X})\to \mathbb{R}~:~\varphi\text{ is non-zero, linear, multiplicative and forward upper semicontinuous}\}.$$
Note that $ \mathcal{S}(\mathcal{X})\subset \mathcal{A}(\mathcal{X})^*$
\end{definition}
\begin{remark}\label{caractcont}
Every $\varphi\in\mathcal{S}(\mathcal{X})$ is in fact continuous. Indeed, to see that $-\varphi$ is upper semicontinuous, we need to give a bound for $-\varphi(f)$ for any $f\in \mathcal{A}(\mathcal{X})$. By denoting the constant function of value $1$ as $\mathds{1}$, and using the fact that $\varphi(\mathds{1})=1$ (since $\varphi$ is non-zero), we have
$$-\varphi(f)=\varphi(-\mathds{1})\varphi(f)\leq \varphi(-\mathds{1})\|f|\|\varphi|^*\leq K \|f|.$$
As a consequence, an equivalent definition of the structure space $\mathcal{S}(\mathcal{X})$ could be given by requiring each functional to be forward continuous instead of usc. 
\end{remark}
\begin{proposition}\label{charpositive}
Every $\varphi\in \mathcal{S}(\mathcal{X})$ is positive.
\end{proposition}
\begin{proof}
We will use the fact that $\mathcal{A}(\mathcal{X})$ is closed under bounded inversions, that is, whenever $f\in \mathcal{A}(\mathcal{X})$ satisfies $f\geq 1$, then $f^{-1}\in \mathcal{A}(\mathcal{X})$. Indeed, if we take $f\in \mathcal{A}(\mathcal{X})$ such that $f\geq 1$, then it is clear that $f^{-1}$ is of class $C^1$. Moreover, we have that the derivative of $-f^{-1}$ at a point $x$ equals $f(x)^{-2}df(x)$, so by Proposition~\ref{normaderivadaslip}, we have that the function $-f^{-1}$ is semi-Lipschitz. It follows that ${f^{-1}\in -SC_b^1(\mathcal{X})\subseteq \mathcal{A}(\mathcal{X})}$.

Using this, and the fact that $\varphi(\mathds{1})=1$, we can show that $\alpha=\varphi(f)$ belongs to $\overline{f(\mathcal{X})}^{|\cdot|}$ for every $f\in \mathcal{A}(\mathcal{X})$. Otherwise, there would exist $\varepsilon >0$ such that no sequence in $f(\mathcal{X})$ accumulates in $B(\alpha, \varepsilon)$, which implies $(f-\alpha)\geq\varepsilon$. Noting that both $(f-\alpha)$ and $(f-\alpha)^2$ belong to $\mathcal{A}(\mathcal{X})$, we arrive at a contradiction: 
$$ 1=\varphi((f-\alpha)^2)\cdot\varphi((f-\alpha)^{-2})=0.$$
\end{proof}
For each $x\in \mathcal{X}$, consider the \emph{evaluation mapping} $\delta_x:\mathcal{A}(\mathcal{X})\to \mathbb{R}$ defined by $\delta_x(f)=f(x)$ for all $f\in \mathcal{A}(\mathcal{X})$. Note that $\delta_x\in \mathcal{S}(\mathcal{X})$ for all $x\in \mathcal{X}$.
\begin{proposition}\label{dense}
The set of evaluation mappings $\{\delta_x \,:\,x\in \mathcal{X}\}$ is dense in $\mathcal{S}(\mathcal{X})$ for the topology of pointwise convergence in $\mathcal{A}(\mathcal{X})^*$ (which will be denoted by $\tau_p$).
\end{proposition}
\begin{proof}

Let $\varphi\in \mathcal{S}(\mathcal{X})$, and consider a basic $\tau_p$-neighborhood of $\varphi$:
$$W=\{\psi\in \mathcal{S}(\mathcal{X})\,:\,|\varphi(f_i)-\psi(f_i)|<\varepsilon\,\,\forall i\in \{1,...,n\}\},$$
\noindent with $\varepsilon>0$, $n\in \mathbb{N}$ and $f_1,...,f_n\in \mathcal{A}(\mathcal{X})$. Suppose that $\delta_x\notin W$ for every $x\in \mathcal{X}$. Then, the function ${g=\sum_{i=1}^{n}(f_i-\varphi(f_i))^2}$ satisfies $g\geq \varepsilon^2$, but $\varphi(g)=0$, which contradicts the positivity of $\varphi$, as ${\varphi(g-\varepsilon^2)=-\varepsilon^2<0}$, since $\varphi(\mathds{1})=1$.
\end{proof}

From this point forward, we will assume $(\mathcal{X},F)$ to be a connected, forward complete and second countable Finsler manifold with Finsler distance $d_\mathcal{X}$. Consider the map $\delta:\mathcal{X}\to (\mathcal{S}(\mathcal{X}),\tau_p)$ associating to each $x\in \mathcal{X}$ the evaluation functional $\delta_x$. It is easy to check that this map is continuous, as every $f\in \mathcal{A}(\mathcal{X})$ is continuous in $\mathcal{X}$. Moreover, we can show that $\delta$ is an embedding.
\begin{proposition}
 The mapping $\delta^{-1}:\delta(\mathcal{X}) \to \mathcal{X}$ that maps each $\delta_x$ to its corresponding $x\in \mathcal{X}$ is $\tau_p$ continuous.
\end{proposition}
\begin{proof}
Consider a net $\left( \delta_{x_\lambda}\right)\subset \delta (\mathcal{X})$ converging to $\delta_{x_0}$ in the pointwise topology of $\mathcal{A}(\mathcal{X})^*$, and consider the function $f(x)=\min\{d_\mathcal{X}(x_0,x),1\}$, which is semi-Lipschitz and bounded, but it fails to be of class $C^1$. To remedy this, for $\varepsilon>0$, consider a smooth approximation $g$ of $f$ obtained using Corollary~\ref{smoothslip}, such that $|\min\{d_\mathcal{X}(x_0,x),1\}-g(x)|\leq \varepsilon $. Then, if we evaluate the net $\left( \delta_{x_\lambda}\right)$ on the function $g\in \mathcal{A}(\mathcal{X})$, we obtain that 
$ \delta_{x_\lambda}(g)=g(x_\lambda)$ converges (in absolute value) to $\delta_{x_0}(g)=g(x_0)\leq \varepsilon$, and since  ${|\min\{d_\mathcal{X}(x_0,x),1\}-g(x)|\leq \varepsilon}$, we conclude that $d_\mathcal{X}(x_0,x_\lambda)$ converges to $0$.
\end{proof}
\begin{proposition}\label{vecnum}
Let $\varphi\in S(\mathcal{X})$. The following are equivalent:
\begin{enumerate}
    \item $\varphi$ has a countable neighborhood basis in $S(\mathcal{X})$ for the pointwise topology $\tau_p$ of $\mathcal{A}(\mathcal{X})^*$.
    \item There exists $x\in \mathcal{X}$ such that $\varphi=\delta_x$.
\end{enumerate}
\end{proposition}
\begin{proof}
$(1)\implies(2)$ Since $(\mathcal{X},F)$ is forward complete, the Hopf-Rinow theorem (see Theorem 6.6.1 of \cite{BCS}) asserts that forward bounded and closed subsets of $\mathcal{X}$ are compact, which implies that the function $f(x)=d_\mathcal{X}(p,x)$ is proper for any $p\in \mathcal{X}$ (recall that a function is called \emph{proper} if the inverse image of every compact set is compact). By taking a $C^1$-smooth semi-Lipschitz approximation of $f$ (using Corollary~\ref{smoothslip}), we obtain a proper semi-Lipschitz function $g\in C^1(\mathcal{X})$. Suppose now that $\varphi\in S(\mathcal{X})\setminus \delta(\mathcal{X})$ has a countable neighborhood basis in $S(\mathcal{X})$ for the pointwise topology. Then, by Proposition~\ref{dense}, there exists a sequence $(x_n)_{n\in \mathbb{N}}\subset \mathcal{X}$, with no convergent subsequences, such that $\delta_{x_n}\to \varphi$ in the pointwise topology. Since $g$ is proper, $\lim g(x_n)=+\infty$, so there exists a subsequence $x_{n_k}$ such that $g(x_{n_k})+1<g(x_{n_{k+1}})$. Now we can choose a $C^1$-smooth function $\theta: \mathbb{R} \to  [0, 1]$, with bounded derivative (so that $\theta \circ g\in SC_b^1(\mathcal{X})_+$), such that $\theta (g(x_{n_{2k+1}}))=1$ and $\theta (g(x_{n_{2k}}))=0$  for every $k\in \mathbb{N}$, which is a contradiction, as the sequence $\delta_{x_{n_k}}(g)$ is not convergent.\\
$(2)\implies(1)$ If $\varphi=\delta_x$ for some $x\in \mathcal{X}$, then consider any countable neighborhood basis $(V_n)$ of $x$ in $\mathcal{X}$. Then the family of closures $\{\mathrm{cl}_{S(\mathcal{X})}\delta(V_n)\,:\,n\in \mathbb{N}\}$ is a countable neighborhood basis of $\delta_x$ in $S(\mathcal{X})$.
\end{proof}
\begin{lemma}\label{metric}
For each $x,y\in \mathcal{X}$, 
$$\min\{d_\mathcal{X}(x,y),1\}\leq\|\delta_y-\delta_x|^*\leq d_\mathcal{X}(x,y),$$
\noindent where $\|\cdot|^*$ is the conic hemi-norm of $\mathcal{A}(\mathcal{X})^*$.
\end{lemma}
\begin{proof}
First, let us note that $\delta_y-\delta_x$ belongs to the dual cone $\mathcal{A}(\mathcal{X})^*$, since for any $f\in\mathcal{A}(\mathcal{X})$ with finite hemi-norm we have $(\delta_y-\delta_x)(f)=f(y)-f(x)\leq \|f|d_\mathcal{X}(x,y)$. It follows that ${\|\delta_y-\delta_x|^*\leq d_\mathcal{X}(x,y)}$. On the other hand, fix $x,y\in\mathcal{X}$ and for $\varepsilon>0$, consider the function $f(u)=\min\{d_\mathcal{X}(x,u),1\}$ and apply Corollary~\ref{smoothslip} to obtain a $C^1$-smooth semi-Lipschitz approximation $g$ of $f$, such that $|g(u)-f(u)|\leq \varepsilon$ for all $u\in \mathcal{X}$ and $\|g|_S\leq \|f|_S+\varepsilon=1+\varepsilon$. Since $g(u)\geq f(u) -\varepsilon$ for all $u\in \mathcal{X}$, it follows that the function $\Tilde{g}:=g+\varepsilon$ satisfies $\Tilde{g}\geq f\geq 0$, which implies $\Tilde{g}\in SC_b^1(\mathcal{X})_+$. By multiplying $\Tilde{g}$ by $(1+\varepsilon)^{-1}$, we ensure that $\|(1+\varepsilon)^{-1}\Tilde{g}|\leq 1$, so 
\begin{align*}
  \|\delta_y-\delta_x|^* &\geq (\delta_y -\delta_x)((1+\varepsilon)^{-1}\Tilde{g})\\
  &=(1+\varepsilon)^{-1}(\Tilde{g}(y)-\Tilde{g}(x))   \\ 
  &\geq (1+\varepsilon)^{-1}(f(y)-\Tilde{g}(x))\\
 &\geq(1+\varepsilon)^{-1}(f(y)-f(x)-\varepsilon)\\
 &=(1+\varepsilon)^{-1}(\min\{d_\mathcal{X}(x,y),1\}-\varepsilon),
\end{align*}
\noindent for every $\varepsilon>0$, and therefore $\|\delta_y-\delta_x|^*\geq \min\{d_\mathcal{X}(x,y),1\}$.
\end{proof}
We are ready to prove the isomorphic version of our main result. 
\begin{theorem}[Non reversible Myers-Nakai Theorem, isomorphic version]\label{mnasymm}
Let $(\mathcal{X},F_\mathcal{X})$ and $(\mathcal{Y},F_\mathcal{Y})$ be connected, second countable and forward complete Finsler manifolds. Then, for any ${T:\mathcal{A}(\mathcal{Y})\to \mathcal{A}(\mathcal{X})}$, the following are equivalent:

\begin{enumerate}
    \item T is an extended asymmetric normed algebra isomorphism.
    \item There exists a bi-semi-Lipschitz diffeomorphism  $h:(\mathcal{X},d_\mathcal{X})\to (\mathcal{Y},d_\mathcal{Y})$ such that $Tf=f\circ h$ for all $f\in \mathcal{A}(\mathcal{Y})$, and the semi-Lipschitz constants of $h$ and $h^{-1}$ are bounded by $\|T|$ and $\|T^{-1}|$, respectively. 
\end{enumerate} 
\end{theorem}
\begin{proof}
$(2)\implies(1)$ is direct. For the remaining implication, assume that ${T:\mathcal{A}(\mathcal{Y})\to \mathcal{A}(\mathcal{X})}$ is an extended asymmetric normed algebra isomorphism. Consider the dual map ${T^*:\mathcal{A}(\mathcal{X})^*\to \mathcal{A}(\mathcal{Y})^*}$ defined in the usual way, so that $\langle T^*\varphi,f\rangle =\langle \varphi,Tf\rangle$ for all $f\in \mathcal{A}(\mathcal{Y})$ and $\varphi\in \mathcal{A}(\mathcal{X})^*$. Clearly, $T^*$ maps continuous multiplicative functionals into continuous multiplicative functionals,
and since $T^*$ is by definition $\tau_p$ to $\tau_p$ continuous, it follows that $T^*$ restricts to an homeomorphism between $\mathcal{S}(\mathcal{X})$ and $S(\mathcal{Y})$. Consider now the natural embeddings $\delta_\mathcal{X}:\mathcal{X}\to \mathcal{S}(\mathcal{X})$ and $\delta_\mathcal{Y}:\mathcal{Y}\to \mathcal{S}(\mathcal{Y})$. By Lemma~\ref{vecnum}, ${T^*(\delta_\mathcal{X}(\mathcal{X}))=\delta_{\mathcal{Y}}(\mathcal{Y})}$. Now, we can define 
$$h=(\delta_{\mathcal{Y}})^{-1}\circ T^* \circ \delta_\mathcal{X}\,:\mathcal{X}\to \mathcal{Y},$$
\noindent which is an homeomorphism from $\mathcal{X}$ onto $\mathcal{Y}$. Moreover, for any $f\in \mathcal{A}(\mathcal{Y})$, and any $x\in \mathcal{X}$ 
$$ Tf(x)=\langle \delta_x, Tf\rangle = \langle T^*\delta_x,f\rangle\\
    =\langle \delta_{h(x)}, f\rangle = f(h(x))=(f\circ h)(x).$$

Noting that any non negative and compactly supported $f\in C^1(\mathcal{Y})$ belongs to $\mathcal{A}(\mathcal{Y})$, we have that $Tf=f\circ h$ belongs to $\mathcal{A}(\mathcal{X})\subseteq C^1(\mathcal{X})$ for any such $f$. Using local charts, we can deduce that $h$ is smooth, and the same can be said about $h^{-1}$. To see that $h$ is bi-semi-Lipschitz, it suffices to prove that it is locally bi-semi-Lipschitz (see Proposition~\ref{locslip}). To this end, fix $p\in \mathcal{X}$. Consider the open neighborhood of $p$ ${U_p=B^s(p,\frac{1}{2})\cap h^{-1}(B^s(h(p),\frac{1}{2}))}$, where $B^s(q,r)$ denotes the ball of center $q$ and radius $r$ for the symmetrized Finsler distance ${d^s(z,w)=\max\{d_F(z,w),d_F(w,z)\}}$. Then, for any $x_1,x_2\in U_p$, we have $d_\mathcal{X}(x_1,x_2)<1$ and $d_{\mathcal{Y}}(h(x_1),h(x_2))<1$. Then, using Lemma~\ref{metric} we have that
\begin{align*}
    d_{\mathcal{Y}}(h(x_1),h(x_2))&\leq \|\delta_{h(x_2)}-\delta_{h(x_1)}|^*\\
    &=\|T^*\delta_{x_2}-T^*\delta_{x_1}|^*\\
    &\leq \|T^*|\|\delta_{x_2}-\delta_{x_1}|^*\\
    &\leq \|T^*|d_\mathcal{X}(x_1,x_2).
\end{align*}

The same argument can be used for $h^{-1}$, obtaining that $h$ is bi-semi-Lipschitz. Finally, a simple computation (the same as in the symmetric case) shows that $\|T^*|=\|T|$, so the proof is complete. 
\end{proof}

One can easily see that Theorem \ref{thm:main} follows as a direct corollary of Theorem~\ref{mnasymm}, as the norms of the isomorphisms $T$ and $T^{-1}$ give precise bounds for the semi-Lipschitz constants of $h$ and $h^{-1}$. This should be compared to Theorem~\ref{mnquasi} (as stated in \cite[Theorem~16]{GJR-13}), where the bounds obtained for $h$ and $h^{-1}$ depend on the indexes of symmetry of the involved manifolds (referred as constants of quasi-reversibility by the authors), and therefore, does not allow for an isometric result unless both indexes of symmetry are $1$ (that is, when both Finsler manifolds are reversible, see Theorem~\ref{mnreversible}). Therefore, Theorem~\ref{mnasymm} can be seen as an improvement of Theorem~\ref{mnquasi}, by improving the bounds obtained for the semi-Lipschitz constants of the diffeomorphism, as well as by extending the result to non-reversible Finsler manifolds. 
\subsection*{Acknowledgment}

This research has been partially supported by the Center of Mathematical Modeling
CMM, grant FB210005 BASAL funds for centers of excellence (ANID-Chile). 
A preliminary version of this work was presented at Seminario del Departamento de Matem\'atica Aplicada, Escuela de Ingenier\'ia Industrial, UNED (Spain) and at Seminario del Departamento Matem\'aticas Aplicadas y An\'alisis Matem\'atico, Universidad Complutense de Madrid (Spain). The author wishes to thank the organizers for their hospitality, as well as to Jes\'us A. Jaramillo, Estibalitz Durand-Cartagena and M. Isabel Garrido for their insightful comments and discussions on the subject, and to the referee for their suggestions, which greatly improved the presentation of the article.

 \bibliographystyle{plain} 
 \bibliography{MN.bib}

\end{document}